\documentclass[12pt,draft]{amsart}
\usepackage[all]{xy}
\usepackage{amsfonts}
\usepackage{amssymb}

\usepackage{microtype}

\usepackage{enumerate}

\usepackage{color}

\renewcommand{\le}{\leqslant}
\renewcommand{\ge}{\geqslant}

\newtheorem{teo}{Theorem}[section]
\newtheorem{lem}[teo]{Lemma}
\newtheorem{prop}[teo]{Proposition}

\theoremstyle{definition}

\newtheorem{rk}[teo]{Remark}

\def\<{\langle}
\def\>{\rangle}
\def\ss{\subset}

\def\a{\alpha}

\def\e{\varepsilon}

\def\r{\rho}

\def\s{\sigma}
\def\t{\tau}

\def\f{{\varphi}}

\def\G{{\Gamma}}

\def\Z{{\mathbb Z}}

\def\Id{\operatorname{Id}}

\def\1{\mathbf 1}
\def\lcm{\operatorname{lcm}}

\def\rwr{\:\mathrm{wr}\:}

\newcommand{\Mat}[4]{\left( \begin{array}{cc}
                            #1 & #2 \\
                            #3 & #4
                      \end{array} \right)}

\def\Fix{\operatorname{Fix}}

\begin{document}

\title{Reidemeister classes in lamplighter type groups}

\author{Evgenij Troitsky}
\thanks{This work is
supported by the Russian Science Foundation under grant 16-11-10018.}
\address{Dept. of Mech. and Math., Moscow State University,
119991 GSP-1  Moscow, Russia}
\email{troitsky@mech.math.msu.su}
\urladdr{http://mech.math.msu.su/\~{}troitsky}


\begin{abstract}
We prove that for any automorphism $\phi$ of the restricted
wreath product $\Z_2 \rwr \Z^k$ and $\Z_3 \rwr \Z^{2d}$
the Reidemeister number $R(\phi)$ is infinite, i.e. these groups
have the property $R_\infty$.

For $\Z_3 \rwr \Z^{2d+1}$ and $\Z_p \rwr \Z^k$, where $p>3$ is prime,
we give examples of automorphisms with finite Reidemeister numbers.
So these groups do not have  the property $R_\infty$.

For these groups and $\Z_m \rwr \Z$, where $m$ is relatively prime
to $6$, we prove the twisted Burnside-Frobenius theorem (TBFT$_f$): 
if $R(\phi)<\infty$, then it is equal to the number of equivalence classes of finite-dimensional
irreducible unitary representations fixed by the action 
$[\rho]\mapsto [\rho\circ\phi]$.
\end{abstract}

\maketitle

\section*{Introduction}
The \emph{Reidemeister number} $R(\phi)$ 
of an automorphism $\phi$ of a (countable discrete) group $G$
is the
number of its \emph{Reidemeister} or \emph{twisted conjugacy classes},
i.e. the classes of the twisted conjugacy equivalence relation:
$g\sim hg\phi(h^{-1})$, $h,g\in G$. Denote by $\{g\}_\phi$ the
class of $g$.

The following two interrelated problems are in the mainstream of
the study of Reidemeister numbers.

In \cite{FelHill} A.Fel'shtyn and R.Hill conjectured that
$R(\phi)$ is equal to the number of fixed points of the
associated homeomorphism $\widehat{\phi}$ of the unitary
dual $\widehat{G}$ (the set of equivalence classes of irreducible
unitary representations of $G$), if one of these numbers
is finite. The action of $\widehat{\phi}$ on the class of a
representation $\rho$ is defined as $[\rho]\mapsto [\rho\circ \phi]$.
This conjecture is called TBFT (twisted Burnside-Frobenius
theorem). 
This statement can be considered as a generalization
to infinite groups and to the twisted case the classical Burnside-Frobenius theorem: the number of conjugacy classes of a finite group is equal to the number of equivalence classes of its irreducible representations. The TBFT conjecture (more precisely, its 
modification TBFT$_f$, taking into account only finite-dimensional
representations)
was proved for polycyclic-by-finite groups in \cite{polyc,feltroKT}.
Preliminary and related results, examples and counter-examples
can be found in \cite{FelHill,FelTro,FelTroVer,FelIndTro,%
ncrmkwb,FelLuchTro,FelTroJGT,TroTwoEx}.

Also A.Fel'shtyn and co-authors 
formulated the second problem
(a historical overview can be found in \cite{FelLeonTro}): 
the problem of description of the class of groups
having the following $R_\infty$ property: $R(\phi)=\infty$ for any automorphism $\phi:G\to G$. Thus, the second problem is in some sense
complementary to the first one: the question about TBFT has no
sense for $R_\infty$ groups (formally having a positive answer).
The property  $R_\infty$ was studied very intensively during the
last years and was proved and disproved for many groups
(see a bibliography overview in \cite{FelLeonTro} and
\cite{FelTroJGT}, and very recent papers
\cite{DekimpeGoncalves2015Abel,SteinTabackWong2015,FelNasy2016JGT} and the literature
therein). For Jiang type spaces the property $R_\infty$ has some direct topological
consequences (see e.g. \cite{GoWon09Crelle}). Relations with group growth are discussed e.g. in \cite{gowon}. Concerning
applications of Reidemeister numbers in Dynamics we refer to \cite{Jiang,FelshB}. 

The $R_\infty$ property was studied for 
the lamplighter group $\Z_2\rwr\Z$ and
some its generalizations being restricted wreath products with
$\Z$ in 
\cite{gowon1,TabackWong2011,SteinTabackWong2015}.
In particular, in \cite{gowon1} it was proved that most part of
groups of the form $\Z_q\rwr\Z$ are not $R_\infty$ groups.
More precisely, it is an $R_\infty$ group if and only if
$(q,6)\neq 0$.
In contrast with this result, in \cite{SteinTabackWong2015}
it is proved that the generalizations $\G_d(q)$
of the lamplighter group always have the $R_\infty$ property
for $d>2$ (for $d=2$ one has $\G_d(q)\cong \Z_q\rwr\Z$).
The groups $\G_d(q)$ admit a Cayley graph isomorphic
to a Diestel-Leader graph $DL_d(q)$.
The lamplighter group and its generalizations attracted a lot
of attention recently, in particular due to its relations with
automata groups, self-similar groups, and branch groups
(see e.g. \cite{BarNeuWoe2008}).

For groups under consideration in the present paper,
even for $k=2$, the situation is much more complicated, because $\Z$ has only one automorphism with finite Reidemeister number, namely $-\Id$, and its square has infinite Reidemeister number, but 
for $\Z\oplus\Z$ we have a lot of automorphisms
with finite Reidemeister numbers, and many of them have finite Reidemeister numbers for all their iterations (see, e.g. \cite{FelshB}).

In the present paper 
we prove that for any automorphism $\phi$ of the restricted
wreath product $\Z_2 \rwr \Z^k$ 
(Theorem \ref{teo:ihbeskonechno})
and $\Z_3 \rwr \Z^{2d}$ (Theorem \ref{teo:casep3})
the Reidemeister number $R(\phi)$ is infinite, i.e. these groups
have the property $R_\infty$.

For $\Z_3 \rwr \Z^{2d+1}$ (Theorem \ref{teo:casep3})
and $\Z_p \rwr \Z^k$, where $p>3$ is prime, 
(Theorem \ref{teo:r_infty_and_not})
we give examples of automorphisms with finite Reidemeister numbers.
So these groups do not have  the property $R_\infty$.

For these groups and $\Z_m \rwr \Z$, where $m$ is relatively prime
to $6$, we prove in Theorem \ref{teo:TBFTforlamplightertype}
the twisted Burnside-Frobenius theorem (TBFT$_f$): 
if $R(\phi)<\infty$, then it is equal to the number of equivalence classes of finite-dimensional
irreducible unitary representations fixed by $\widehat{\phi}:\rho\mapsto\rho\circ\phi$.

This gives (probably first)
examples of finitely generated residually finite
but not almost polycyclic groups
(having infinitely generated subgroup, \cite[p.~4]{DSegalPoly}), for which the TBFT is true.

\medskip
\textsc{Acknowledgement:} 
The author is indebted to A.~Fel'shtyn for helpful 
discussions in the Max-Planck Institute for
Mathematics (Bonn) in February, 2017  and the MPIM for 
supporting this visit.

The author is grateful to L.~Alania, R.~Jimenez Benitez, and V.~Manuilov for 
valuable advises and suggestions. 

This work is supported by the Russian Science Foundation under grant 16-11-10018.

\section{Preliminaries}\label{sec:prelimi}

The following easy statement is well known:
\begin{prop}\label{prop:epim}
Suppose, $H$ is a $\phi$-invariant normal subgroup of $G$
and $\overline{\phi}:G/H \to G/H$ is the induced automorphism. Then
$\phi$ induces an epimorphism of each Reidemeister class
 of $\phi$ onto some Reidemeister class of 
$\overline{\phi}$. In particular, one has $R(\overline{\phi})\le R(\phi)$. 
\end{prop}

Denote by $C(\phi)$ the fixed point subgroup.
The following much more non-trivial statement can be extracted
from \cite{go:nil1}   (see also \cite{polyc}):
\begin{lem}\label{lem:cherezfixed}
In the above situation
$R(\phi|_H)\le R(\phi)\cdot |C(\overline{\phi})|.$
\end{lem}
 
It is well known (see \cite{FelTro}) the following. 
\begin{lem}\label{lem:abelcase} 
For an abelian group $G$ the Reidemeister class of the unit element
is a subgroup, and the other classes are corresponding cosets. 
\end{lem} 
 
The following statement is very useful in the field. 
\begin{lem}\label{lem:innereqrei}
A right shift by $g\in G$ maps Reidemeister classes of $\phi$
onto Reidemeister classes of $\t_{g^{-1}}\circ \f$, where
$\t_g$ is the inner automorphism: $\t_g(x)=gxg^{-1}$.
In particular, $R(\t_g\circ \phi)=R(\phi)$.
\end{lem} 

\begin{proof}
Indeed,
$$
xy\f(x^{-1})g=x (yg) g^{-1}\f(x^{-1})g=x (yg)(\t_{g^{-1}}\circ\f)(x^{-1}).
$$
\end{proof}

Also we need the following statement (\cite{FelTroResFin}, \cite[Prop.~3.4]{FelLuchTro}):

\begin{lem}\label{lem:modjab} 
Let $\phi:G\to G$ be an automorphism of a finitely generated residually
finite group $G$ with $R(\phi)<\infty$ $($in particular, $G$ can be a finitely generated 
abelian group$)$. Then the subgroup of fixed elements is finite:
$|C(\phi)|<\infty$.
\end{lem}

Note, that this is not correct for infinitely generated groups, see \cite{TroTwoEx}.

Combining this lemma with some results of \cite{go:nil1}
one can prove:
\begin{lem}\label{lem:nessuffinRei}
Suppose in the situation of Lemma \ref{prop:epim} that
$G/H$ is a finitely generated residually finite group.
Then $R(\phi)<\infty$ if and only if $R(\overline{\phi})<\infty$
and $R(\tau_g \phi')<\infty$ for any $g\in G$.
\end{lem}

\section{The case of $\Z_2\rwr\Z^k$}\label{sec:wreathex}
 
Let $\G:=\Z_2 \:\mbox{wr}\: \Z^k$ be a restricted wreath product. In other words,  
$$\G=\oplus_{x\in \Z^k} (\Z_2)_{(x)}\rtimes_\alpha \Z^k, 
\qquad (\Z_2)_{(x)}\cong \Z_2,
\qquad
\a(y)(\delta_{x}):=\delta_{y+x},
$$
where $y\in \Z^k$ and $\delta_{x}$ is a unique non-trivial element of 
$(\Z_2)_{(x)}\ss\G$. The direct sum supposes only finitely many non-trivial components for each element (in contrast with the direct product corresponding to the (unrestricted) wreath product).

The group $\G$ is a finitely generated metabelian group, in particular,
residually finite (see e.g. \cite{Robinson}).

Let $\phi:\G\to\G$ be an automorphism. We will prove that $R(\phi)=\infty$.
Denote $\Sigma:=\oplus_{x\in \Z^k} (\Z_2)_{(x)}\subset \G$. Then $\Sigma$ is
a characteristic subgroup as the torsion subgroup. 
Denote the restriction of $\phi$ by $\phi':\Sigma \to
\Sigma$, and the quotient automorphism by $\overline{\phi}: \Z^k\to \Z^k$. 

If $R(\phi)<\infty$, then $R(\overline{\phi})<\infty$ 
by Proposition \ref{prop:epim}.
Hence, by Lemma \ref{lem:modjab},
$\overline{\phi}$ has finitely many fixed elements. Thus, by Lemma \ref{lem:cherezfixed},
$R(\phi')<\infty$. Hence, to prove that $R(\phi)=\infty$, it is sufficient to verify
that $R(\phi')=\infty$.

Since $\Sigma$ is abelian, the results of e.g. \cite{Curran2008} imply that
\begin{equation}\label{eq:relationforautom}
\phi'(\alpha(g)(h))=\alpha(\overline{\phi}(g))(\phi'(h)),\qquad h\in\Sigma,\quad
g\in\Z^k.
\end{equation}

Any element of $\Sigma$ is a finite sum of some elements $\delta_{x}$. Let
\begin{equation}\label{eq:phiotnulia}
\phi'(\delta_{0})=\delta_{x(1)}+\dots+\delta_{x(n)}.
\end{equation}

The following lemma generalizes \cite[Prop. 2.1]{gowon1}
from the case $k=1$ to arbitrary $k$.

\begin{lem}\label{lem:odinobraz}
In {\rm (\ref{eq:phiotnulia})} one has $n=1$. Moreover, $\phi'$ is a permutation of
$\delta_{x}$'s.
\end{lem}

\begin{proof}
First of all, apply (\ref{eq:relationforautom}) to $h=\delta_{0}$.
We have:
\begin{equation}\label{eq:firstfor00}
\phi'(\delta_{g})=\phi'(\alpha(g)(h))=\alpha(\overline{\phi}(g))(\phi'(\delta_{0})).
\end{equation}
Thus, for any $g\in \Z^k$, the element $\phi'(\delta_{g})$ is obtained by the appropriate shift of indexes in the right side expression in (\ref{eq:phiotnulia}). 

Now suppose that $n\neq 1$, and $\phi'(h)=\delta_{0}$
for some $h=\delta_{r(1)}+\dots+\delta_{r(t)}$. 
Then $t\neq 1$, because the statement of the lemma for $\phi'$
and its inverse are equivalent. Denote by $T\subset \Z^k$
the support of $\phi'(\delta_{0})$, i.e., 
$$
T=\{x(1),\dots, x(n)\}.
$$
Denote by $T_1,\dots,T_t$ the supports of  
$\phi'(\delta_{r(1)}),\dots,\phi'(\delta_{r(t)})$
respectively. They are appropriate distinct shifts of $T$.
Denote $S:=T_1\cup\dots\cup T_t$ (without cancellations).
After cancellations in $\Sigma$ (i.e. excluding of points
in $S$ covered by an even number of $T_j$'s) we should obtain only
one point, namely, $0$.

Introduce now the notion of 
$(\e_1 \s_1,\e_2 \s_2,\dots,\e_k \s_k)$-\emph{vertex} $(\e_1 \s_1,\e_2 \s_2,\dots,\e_k \s_k)[R]\in R$
for any bounded subset $R\ss \Z^k$, where $(\s_1,\dots,\s_k)$ is
a permutation of $(1,\dots,k)$ and $\e_i=\pm 1$. We define it
inductively in the following way: $R_{k-1}$ is the subset of
points of $R$ with minimal (if $\e_1=-1$) or maximal (if $\e_1=+1$)
coordinate number $\s_1$, $R_{k-2}$ is the subset of
points of $R_{k-1}$ with minimal (if $\e_2=-1$) or maximal 
(if $\e_2=+1$)
coordinate number $\s_2$, and so on. Then $R_0$ is one point.
This point we define to be 
$(\e_1 \s_1,\e_2 \s_2,\dots,\e_k \s_k)[R]$.
This point also can be considered as a lexicographic 
maximum of points of $R$
for the ordering $\sigma$ and the inverse direction of
that coordinates, where $\s_j=-1$, i.e. a
\emph{lexicographic maximum with respect to} $(\e_1 \s_1,\e_2 \s_2,\dots,\e_k \s_k)$.

Evidently, $(\e_1 \s_1,\e_2 \s_2,\dots,\e_k \s_k)[T_j]$
is a $\Z^k$-shift of $(\e_1 \s_1,\e_2 \s_2,\dots,\e_k \s_k)[T]$
and $T_j=T_i$ if and only if
$$
(\e_1 \s_1,\e_2 \s_2,\dots,\e_k \s_k)[T_j]=
(\e_1 \s_1,\e_2 \s_2,\dots,\e_k \s_k)[T_i]
$$
for at least one (thus, for any)  $(\e_1 \s_1,\e_2 \s_2,\dots,\e_k \s_k)$. Hence, in our situation, they are distinct.

We claim that for any $(\e_1 \s_1,\e_2 \s_2,\dots,\e_k \s_k)$
the vertex $(\e_1 \s_1,\e_2 \s_2,\dots,\e_k \s_k)[S]$
coincides with $(\e_1 \s_1,\e_2 \s_2,\dots,\e_k \s_k)[T_j]$
for one and only one $j$ and is not covered by other points. 
Indeed, the uniqueness follows from the argument above.
If it is covered by some point of $T_j$ other than  
$(\e_1 \s_1,\e_2 \s_2,\dots,\e_k \s_k)[T_j]$, then it would be
not the lexicographic maximum w.r.t.  
$(\e_1 \s_1,\e_2 \s_2,\dots,\e_k \s_k)$, because
$(\e_1 \s_1,\e_2 \s_2,\dots,\e_k \s_k)[T_j]\in S$
would be greater.

Thus no vertex will be canceled. Thus they all coincide  
with $0$ and $S=\{0\}$. Hence, $n=r=1$.

Together with the argument at the beginning of the proof, this
gives the second statement.
\end{proof}

By this lemma, we can define $x_0 \in \Z^k$ by 
$\phi'(\delta_{0})=\delta_{x_0}$.
Equation (\ref{eq:firstfor00}) can be written now as
\begin{equation}\label{eq:secondfor00}
\phi'(\delta_{y})=\delta_{y'},\qquad y':=\overline{\phi}(y)+x_0\in \Z^k.
\end{equation}

\begin{lem}\label{lem:razlichnyeklassy}
If $\delta_{x_1}$ and $\delta_{x_2}$ belong to the same Reidemeister
class of $\phi'$, then
\begin{equation}\label{eq:condfortwconj}
\overline{\phi}^t(x_1)+ \overline{\phi}^{t-1}(x_0)+\dots + \overline{\phi}(x_0)+ x_0= x_2
\end{equation}
or
\begin{equation}\label{eq:condfortwconj1}
\overline{\phi}^t(x_2)+ \overline{\phi}^{t-1}(x_0)+\dots + \overline{\phi}(x_0)+ x_0= x_1
\end{equation}
for some integer $t$.
\end{lem}

\begin{proof}
By Lemma \ref{lem:abelcase},
the elements $\delta_{x_1}$ and $\delta_{x_2}$ belong to the same Reidemeister
class of $\phi'$ if and only if $\delta_{x_1}-\delta_{x_2}=
h-\phi'(h)$ for some $h\in \Sigma$. Representing $h$ as $h=\delta_{u(1)}+\dots +
\delta_{u(t)}$ (with distinct summands)
 and applying (\ref{eq:secondfor00}) one has
$$
\delta_{x_1}-\delta_{x_2}=h-\phi'(h)=\sum_{j=1}^t [\delta_{u(j)}-\delta_{u(j)'} ].
$$
This is the same in $\Sigma$ as
$$
\delta_{x_1}+\delta_{x_2}=\sum_{j=1}^t [\delta_{u(j)}+\delta_{u(j)'} ].
$$

Since  all $\delta_{u(j)}$ are distinct, 
all $\delta_{u(j)'}$ are distinct
too, by Lemma \ref{lem:odinobraz}. 
So the cancellation on the right
can be only when $\delta_{u(j)}=\delta_{u(i)'}$.
So one of $\delta_{u(j)}$ should be equal to $\delta_{x_1}$,
one of $\delta_{u(i)'}$ should be equal to $\delta_{x_2}$ (or vice versa), and the remaining $\delta$'s should annihilate. 
Thus, after the appropriate renumbering of $1,\dots,t$, we have
in the first case:
$$
x_1=u(1),\quad u(1)'=u(2),\quad \dots\quad
u(t-1)'=u(t),\quad
u(t)'=x_2,
$$
or
\begin{eqnarray*}
\overline{\phi}(x_1)+ x_0&=&u(2),\\
\overline{\phi}^2(x_1)+ \overline{\phi} (x_0)+ x_0&=&u(2)'=u(3),\\
\overline{\phi}^3(x_1)+  \overline{\phi}^2 (x_0)+ \overline{\phi} (x_0)+ x_0&=&u(3)'=u(4),\\
\dots&\dots&\dots\\
\overline{\phi}^t(x_1)+\overline{\phi}^{t-1}(x_0)+ \dots+\overline{\phi} (x_0)+ x_0&=&u(t)'= x_2.
\end{eqnarray*}

In the second case we need to interchange $x_1$ and $x_2$:
$$
\overline{\phi}^t(x_2)+\overline{\phi}^{t-1}(x_0)+ \dots+\overline{\phi} (x_0)+ x_0= x_1.
$$
\end{proof}

\begin{teo}\label{teo:ihbeskonechno}
The group $\G=\Z_2\rwr\Z^k$ has the property $R_\infty$.
\end{teo}

\begin{proof}
One can reduce the proof of $R(\phi)=\infty$ to the case $x_0=0$.
Indeed, consider the element $w:=-x_0\in\Z^k\subset \G$
and the corresponding inner automorphism $\tau_w:\G\to \G$.
Then by Lemma \ref{lem:innereqrei}, $R(\tau_w\circ\phi)=
R(\phi)$. On the other hand, by the definition of a semidirect
product,
$$
(\tau_w\circ\phi)'(\delta_{0})=\alpha(w)(\phi'(\delta_{0}))=
\alpha( -x_0)(\delta_{x_0})= \delta_{0}.
$$

So, suppose $x_0=0$. Then (\ref{eq:condfortwconj}) 
and (\ref{eq:condfortwconj1})
take the form $ \overline{\phi}^t(x_1)= x_2$ for some integer $t$.
Thus, it is sufficient to prove that $ \overline{\phi}:\Z^k\to \Z^k$ has infinitely many
orbits.

For this purpose denote by $A\in GL_k(\Z)$ the matrix of $ \overline{\phi}$.
Let us show that each orbit intersects the first coordinate axis not more than in 2 points.
Denote by $(x,0,\dots,0)=x\cdot e_1$, $x\ne 0$, one point from the intersection, and let $\overline{\phi}^n(x \cdot e_1)=
x \cdot \overline{\phi}^n(e_1)=x\cdot y\cdot e_1$ be another
intersection point.
Since $A^n$ as an element of 
$GL_k(\Z)$ has the first column $(r_1,\dots,r_k)$
such that $\gcd(r_1,\dots,r_k)=1$ (because
the expansion of the determinant by the first column
has the form $\pm 1=r_1\cdot R_1+\cdots +r_k\cdot R_k$).
Thus, $y=\pm 1$ and $xy=\pm x$.
\end{proof}

\section{The case of $\Z_p\rwr \Z^k$ for $p>3$}\label{sec:pge5}
A part of argument in this section will be close to some argument of 
\cite{gowon1}. Suppose now that $\G=\Z_p \rwr \Z^k$ for a general
prime $p$. We conserve the notation $\Sigma$ for the normal subgroup 
$\oplus \Z_p$.

Now $\delta_x$ is a generator of a subgroup $A_x$ isomorphic
to $\Z_p$, and $p\cdot \delta_x=0$.
Suppose
\begin{equation}\label{eq:phiotnulia_m}
\phi'(\delta_{0})=m_1 \delta_{x(1)}+\dots+m_n \delta_{x(n)}.
\end{equation} 
Then as above,
\begin{equation}\label{eq:phi_m}
\phi'(\delta_{x})=m_1 \delta_{\overline{\phi}(x)+ x(1)}+\dots+m_n \delta_{\overline{\phi}(x)+ x(n)}.
\end{equation}

First of all we need an analog of Lemma \ref{lem:odinobraz}.

In the general situation instead of the sets $T_j$
we need $(T_j,s_j\cdot \vec{m})$, which are some shifts 
with multiplication of
$(T,\vec{m})$, where $\vec{m}=(m_1,\dots,m_n)$.
Fortunately (and that is why we have restricted ourselves
to the prime order case) if $m\in\Z_p$, $m\neq 0$, then it
generates $\Z_p$. That is why the sum of
several elements with the same
$T_{j(1)}=\cdots=T_{j(r)}$ either has the same support 
$=T_{j(1)}$
and coefficients vector $(m_{j(1)}+\cdots+m_{j(r)})\vec{m}$,
or completely annihilates, when $m_{j(1)}+\cdots+m_{j(r)}=0\mod p$. 
  
So, after cancellations we may assume that all supports $T_j$
are distinct and repeat the remaining part of the proof of
Lemma  \ref{lem:odinobraz} and obtain

\begin{lem}\label{lem:odinobrazprime}
If $p$ is prime, one has 
$$
\phi'(\delta_0)=m\cdot \delta_{x_0}
$$
for some $x_0 \in\Z^k$ and $0\neq m\in \Z_p$. 
\end{lem}

Let us note, that generally $m\ne 1$ in this situation. For
example, in $\Z_3$ one can take $s=m=2$ and $sm=4=1\mod 3$.

For other elements we have
\begin{equation}\label{eq:phi_m_mod}
\phi'(\delta_{x})=m \delta_{\overline{\phi}(x)+ x_0}.
\end{equation}

To calculate $R(\phi')$ we need to calculate the
index of the image of $(1-\phi')\Sigma$ in $\Sigma$.

Suppose first that $x_0=0$,
$$
\phi'(\delta_{x})=m \delta_{\overline{\phi}(x)}.
$$

Then for any $x$ we have $(1-\phi')$-invariant
subgroup
\begin{equation}\label{eq:invsubgroup}
\cdots \oplus A_{\overline{\phi}^{-1}(x)} \oplus
A_x\oplus A_{\overline{\phi}(x)} \oplus A_{\overline{\phi}^2(x)}
\oplus \cdots.
\end{equation}
In contrast with the case $k=1$ considered in \cite{gowon1},
the corresponding orbit of $\overline{\phi}$ can be infinite
or finite, but not necessary of length $2$. 

Let us note that since $\overline{\t_g \phi} =
\overline{\phi}$, they have the same orbit structure.

\begin{lem}\label{lem:infiniteforinfiniteorbit}
If an orbit is infinite, then the corresponding restriction
of $1-\phi'$ on the subgroup (\ref{eq:invsubgroup}) 
is not an epimorphism.
\end{lem}

\begin{proof}
Indeed, under the appropriate description,
$$
1-\phi': (\dots,0,0,a_1,a_2,\dots,a_r,0,0,\dots)
\mapsto (\dots,0,-ma_1,a_1-ma_2,a_2-ma_3,\dots, a_r,0,0,\dots)
$$
If $a_1\neq 0$ and $a_r\neq 0$, then $-ma_1\neq 0$ and
the length of non-trivial part increases. Thus, elements
concentrated in one summand, e.g. $\delta_x$, are not in the image.
\end{proof}

If the orbit is finite of length $s$, the matrix of $(1-\phi')$
has the form 
\begin{equation}\label{eq:matrix_for_fin}
E-M=\left(\begin{array}{cccccc}
1   &0  &\cdots &                    &0  & -m \\
-m & 1 &  0&                     &  & 0  \\ 
0  & -m&  1 &  0&                        \\
0  & 0 &  -m & 1&            \ddots      &\vdots  \\
\vdots & &\ddots & \ddots & \ddots   & 0   \\
0      &\cdots & \cdots &                    0  &  -m & 1
\end{array}
\right)
\end{equation}
and is an epimorphism if and only if its determinant  
is not zero (for prime $p$):
\begin{equation}
\det (E-M) = 1-m^s\not\equiv 0\mod p.
\end{equation}

\begin{lem}\label{lem:inf_many_orbits}
For a  non-trivial orbit of $\overline{\phi}:\Z^k\to \Z^k$, 
there exist infinitely many orbits of the same cardinality
(finite or infinite).

If $R(\overline{\phi})<\infty$ there is a unique trivial orbit.
\end{lem}

\begin{proof}
Let $(r_1,\dots,r_k)\neq 0$ be a point of the orbit. 
Since $\overline{\phi}$ as an element of 
$GL(\Z,k)$ preserves $\gcd(r_1,\dots,r_k)$, the elements
$(i\cdot r_1,\dots,i\cdot r_k)$, $i\in\Z$, $i>2$, 
belong to distinct orbits of the same cardinality.

If there is a non-zero fixed point $x_*$ of $\phi'$, then there is
an infinite series $\{s\cdot x_*\}$, $s\in \Z$, of fixed points.
By Lemma \ref{lem:modjab} this contradicts to 
$R(\overline{\phi})<\infty$.
\end{proof}

\begin{lem}\label{lem:reidem_num_1_or_inf}
If $\overline{\phi}$ has an infinite orbit, $R(\phi')=\infty$
and $R(\phi)=\infty$.

If $\overline{\phi}$ has only finite orbits, there are two
possibilities:
\begin{enumerate}[\rm 1)]
\item $R(\phi')=R(\phi)=\infty$. This occurs if, at least
for one orbit, the corresponding restriction of $1-\phi'$
is not an epimorphism.
\item $R(\phi')=1$. This occurs if, for all orbits, 
the corresponding restriction of $1-\phi'$
is an epimorphism. 
\end{enumerate}
If we have one of this cases for $\phi'$, then the same is
true for all $\t_g\circ \phi'$.
\end{lem}

\begin{proof}
Let us note that $\t_g = \t_{g'}:\Sigma\to\Sigma$
if $g^{-1}g'\in \Sigma$, so all automorphisms
$\t_g\circ\phi':\Sigma\to\Sigma$ are described by $g\in \Z^k$.
In this case $\t_g \circ\phi'=\a(g)\circ \phi'$
and the sizes of above invariant groups are the same as for
$\phi'$.

As it was explained above all possible automorphisms
$\phi'$ corresponding to a given $\overline{\phi}=
\overline{\t_g\circ \phi}$ are completely
defined by $m$ and $x_0$ such that 
$\phi'(\delta_0)=m\cdot\delta_{x_0}$, and each pair $(m,x_0)$
defines some $\phi'$ and $\phi$. The relation
$$
\t_{y_0} \circ\phi'(\delta_0) =\a(y_0)\circ \phi'(\delta_0)
= m\cdot\delta_{y_0+x_0}
$$
shows that all automorphisms of $\Sigma$ with the same $m$
differ from each other by an appropriate $\t_{y_0}$
and vice versa.

Hence, if $\overline{\phi}$ has an infinite orbit
and we take for $\phi$, some $g\in \Z^k$ 
such that $\t_g \circ \phi'(\delta_0)=m\cdot\delta_0$,
then there exists an appropriate invariant subgroup of $\phi'$ 
(over this orbit) with a non-epimorphic
restriction of $1-\t_g \circ \phi'$ (Lemma \ref{lem:infiniteforinfiniteorbit}). Since we have
infinitely many such orbits (Lemma \ref{lem:inf_many_orbits})
then $R(\phi')=\infty$ and $R(\phi)=\infty$. 

Now we will describe the matrix form 
of restrictions onto invariant subgroups of
an arbitrary $\phi'$ (i.e. not necessary $x_0=0$)
for the case of finite orbits
of $\overline{\phi}$.
Suppose, $\phi'(\delta_0)=m\delta(x_0)$, $s$ is the
length of the orbit $x_1,\overline{\phi}(x_1),\dots,\overline{\phi}^{s-1}(x_1)$, 
$t$ is the
length of the orbit $x_0,\overline{\phi}(x_0),\dots,\overline{\phi}^{t-1}(x_0)$, 
we have
$$
\phi'(\delta_{x_1})=\alpha(\overline{\phi}(x_1)) \phi'(\delta_0)=
m\cdot \delta_{\overline{\phi}(x_1)+x_0},
$$
$$
\phi'(\delta_{\overline{\phi}(x_1)+x_0})=\alpha(\overline{\phi}(\overline{\phi}(x_1)+x_0)) \phi'(\delta_0)=
m\cdot \delta_{\overline{\phi}^2(x_1)+\overline{\phi}(x_0)+x_0},
$$
$$
\phi'(\delta_{\overline{\phi}^2(x_1)+\overline{\phi}(x_0)+x_0})=\alpha(\overline{\phi}(\overline{\phi}^2(x_1)+\overline{\phi}(x_0)+x_0)) \phi'(\delta_0)=
m\cdot \delta_{\overline{\phi}^3(x_1)+\overline{\phi}^2(x_0)+\overline{\phi}(x_0)+x_0},\dots
$$
In order to estimate the length $r$ of the underlying orbit
\begin{equation}\label{eq:formofgenorb}
x_1 \mapsto \overline{\phi}(x_1)+x_0 \mapsto
\overline{\phi}^2(x_1)+\overline{\phi}(x_0)+x_0
\mapsto \overline{\phi}^3(x_1)+\overline{\phi}^2(x_0)+\overline{\phi}(x_0)+x_0 \mapsto \cdots
\end{equation}
remark that 
\begin{equation}\label{eq:prop_of_t}
y:=\overline{\phi}^{t-1}(x_0)+\overline{\phi}^{t-2}(x_0)+\cdots +x_0=0.
\end{equation}
Indeed,
$$
\overline{\phi}(y)-y= \overline{\phi}^{t}(x_0)-x_0=0.
$$
Hence, if $y\ne 0$, we have a non-trivial fixed point for 
$\overline{\phi}$. This contradicts $R(\overline{\phi})<\infty$
(as in the proof of Lemma \ref{lem:inf_many_orbits}). 

Equality (\ref{eq:prop_of_t}) and the definition of $s$
imply
\begin{equation}\label{eq:gener_orbit}
x_1=\overline{\phi}^r(x_1)+\overline{\phi}^{r-1}(x_0)+\cdots +\overline{\phi}(x_0)+x_0, \mbox{ where }r=\lcm(s,t).
\end{equation}
 
Suppose now that $\overline{\phi}$ does not have infinite orbits.
Then the length of any orbit is bounded by
$\lcm(l_1,\dots,l_k)$, where $l_i$ is the length of the orbit of
the element $e_i$ of the standard base. More precisely, any
orbit length is a divisor of $\lcm(l_1,\dots,l_k)$.
In particular, for any $s$ and $t$, as above,
$r=\lcm(t,s)$ is a divisor of $\lcm(l_1,\dots,l_k)$.
Hence, by (\ref{eq:gener_orbit}) the length $r'$ of 
(\ref{eq:formofgenorb}) is some divisor of $r$ and so of
$\lcm(l_1,\dots,l_k)$.

If the underlying orbit (\ref{eq:formofgenorb}) 
does not start from the point $x_1=0$, one
can obtain infinitely many underlying orbits 
by multiplying $x_1$ by different positive integers $j=1,2,$. 
More precisely, some of them can coincide, but for a sufficiently
large $j$, the point $j x_1$ will not be an element of the orbit 
(\ref{eq:formofgenorb}), etc. So, there is infinitely many
distinct orbits. Moreover, we can find a sufficiently large $J$ 
such that for any $j\ge J$, the distances between
$j x_1$ and $\overline{\phi}(j \cdot x_1)= j\cdot\overline{\phi}(x_1)$
are more than $x_0$, $\overline{\phi}(x_0)+x_0$, \dots
$\overline{\phi}^{t-1}(x_0)+\cdots +\overline{\phi}(x_0)+x_0$.
Evidently, for these orbits the length $r'=r$ (not only a divisor
of $r$). So among the orbits starting in $x_1$, $2x_1$, \dots,
we have infinitely many orbits of length $r$ and
some finite number $< J$  of orbits of some length dividing $r$. 
Then the Reidemeister number of the restriction on the subgroup
over the union of these orbits is finite if and only if
$1-m^r \not\equiv 0\mod p$. But then for any divisor $r'$ of $r$
we have, for $r'':=r/r'$,
$$
1-m^r = (1-m^{r'})(1+m^{r'}+m^{2r'}+\cdots +m^{r'-1})\not\equiv 0\mod p.
$$
Hence,  $1-m^{r'} \not\equiv 0\mod p$. Thus, $1-\phi'$ is an
epimorphism over the orbits from the above finite series too.
In particular, for the ``initial'' orbit (\ref{eq:formofgenorb}).

It remains to discuss the case of $x_1=0$. 
In this case the length of the orbit is $t$. Considering  
some $x_1\ne 0$ with the length of $\overline{\phi}$-orbit equal to
some $s\ne 0$ we arrive as above to  
$1-m^r \not\equiv 0\mod p$, where $r=\lcm(s,t)$.
Since $t$ divides
$r=\lcm(s,t)$, we obtain $1-m^t \not\equiv 0\mod p$ 
similarly to the case of $r'$ above. Thus $1-\phi'$
is an epimorphism over this orbit too. 
Hence, it is an epimorphism in entire $\Sigma$.

As it was explained in the beginning of the proof,
to vary $\phi'$ is the same as to consider various
$\t_g\circ \phi'$ for a fixed $\phi'$. 
This completes the proof.
\end{proof}

\begin{teo}\label{teo:r_infty_and_not}
If $p>3$, we can find an authomorphism with finite Reidemeister
number. Thus,
$\Z_p \rwr \Z^k$ does not have the property $R_\infty$ if $p>3$.
\end{teo}

\begin{proof}
For $p>3$ consider $\overline{\phi}=-\Id$. Then
$R(\overline{\phi})=2^k$ and all non-trivial orbits are
of length $2$. Define $\phi'$ by $x_0=0$ and some non-zero
$m\in \Z_p$,
satisfying $1-m^2\not\equiv 0\mod p$.
Take e.g. $m=2$ (cf. \cite[p. 879]{gowon1}). Then
$1-m^2=-3\not\equiv 0\mod p$ for any prime $p>3$.
\end{proof}

The intermediate case of $p=3$ will be studied in the next
section and the answer will depend on parity of $k$.

\section{The case of $\Z_3\rwr \Z^k$}\label{sec:case_p3}

\begin{teo}\label{teo:casep3}
The group $\G=\Z_3\rwr \Z^k$ has the property $R_\infty$
for odd $k$ and does not have 
the property $R_\infty$ for even $k$.
\end{teo}

\begin{proof}
By Lemma \ref{lem:infiniteforinfiniteorbit}, $R(\phi)$ may be
finite only if all orbits of $\overline{\phi}$ are finite.

First, note that in this case $m$ can be equal to $1$ or $2$.

If $m=1$, then $1-m^s\equiv 0\mod 3$ not depending on 
the length $s$ of the corresponding orbit. Keeping in mind
the argument from the previous section, in particular,
Lemma \ref{lem:inf_many_orbits}, we obtain $R(\phi)=\infty$.

If $m=2$, then $1-m^r\equiv 0\mod 3$ for even $r$ and
$1-m^s\not\equiv 0\mod 3$ for odd $r$.

Denote by $M\in GL(k,\Z)$ the matrix of $\overline{\phi}$.
Since all orbits are finite, in particular, the orbits of
the elements of the standard base, we have $M^r=E$ for some $r\in \Z$, $r>1$ (see the previous section for more detail).

If $k=2d+1$, then $\det(M-\lambda E)=0$ has at least
one real solution $\lambda_0$. It must be a root of 1 of degree $r$.
Thus, $\lambda_0=\pm 1$. If $\lambda_0=1$, then $\det(E-M)=0$ and
$R(\overline{\phi})=\infty$. If $\lambda_0=-1$, then $r$ is even,
and $\overline{\phi}$ has an orbit of even length $s=2m$. 
Let $t$ be the length of the $\overline{\phi}$-orbit of $x_0$,
where $\phi'(\delta_0)=\delta_{x_0}$. 
Then, as in the proof of Lemma \ref{lem:reidem_num_1_or_inf},
we can detect infinitely many orbits of length $r=\lcm(s,t)$.
Since $r$ is even, the restriction of $1-\phi'$ on the
subgroup, related the underlying orbit of the form 
(\ref{eq:formofgenorb}) is not an epimorphism.
Since there is infinitely many such orbits, $R(\phi')=\infty$.
Then $R(\phi)=\infty$, as above. 

If $k=2$, consider $M$ to be the generator $\Mat 0 1 {-1} {-1}$
of a subgroup in $GL(2,\Z)$ isomorphic to $\Z_3$ (see, e.g.
\cite[p.~179]{Newman1972book}). For this $M$, $R(\overline{\phi})=
\det(E-M)=3$. It has only orbits of length 3 (except of the trivial
one). The same is true for $\underbrace{M\oplus\cdots\oplus M}_{d}$
in $\Z^{2d}$ with $R(\overline{\phi})=3^d$.
Then the lengths of all underlying orbits are some powers of $3$ 
(except maybe of the trivial one) and $1-\phi'$ is an epimorphism
(for $\phi$ defined by this $\overline{\phi}$, $m=2$, 
and arbitrary $x_0$).
As in the previous section, this means that all $\t_g\circ\phi'$
have $R(\t_g\circ\phi')=1$ and $R(\phi)=R(\overline{\phi})<\infty$.
\end{proof}

\begin{rk}
If $k=1$, in particular, is odd, we obtain the $R_\infty$
property for $\Z_3\rwr \Z$. This is a particular case of
\cite{gowon1}.
\end{rk}

\section{Twisted Burnside-Frobenius theorem}\label{sec:TBFT}

\begin{teo}\label{teo:TBFTforlamplightertype}
Suppose, $\G$ is $\Z_m \rwr \Z$, where $m$ is relatively prime
to $6$,
or $\Z_p \rwr \Z^k$ for a prime $p>3$ and an arbitrary $k$,
or $\Z_3 \rwr \Z^{2d}$. 
Then the twisted
Burnside-Frobenius theorem is true for $\G$.
\end{teo}

\begin{proof}
In all these cases for $\phi$ with $R(\phi)<\infty$, we
have $R(\t_g\circ \phi')=1$ (for prime cases this is proved
in Lemma \ref{lem:reidem_num_1_or_inf}, for $\Z_m\rwr\Z$
a similar statement can be easily extracted from
the proof of Theorem 2.3 in \cite{gowon1}).

Evidently, $\{h\}_{\phi'}\subset \{h\}_{\phi}$, $h\in\Sigma$.
Thus, $R(\phi')=1$ implies that only one Reidemeister class
of $\phi$ is mapped onto the class $\{e\}_{\overline{\phi}}$.
By Lemma \ref{lem:innereqrei} $R(\t_g\circ \phi')=1$ implies
the same for other classes $\{g\}_{\overline{\phi}}$.
Thus, $\pi: \Z_p \rwr \Z^k \to \Z^k$ induces a bijection of
Reidemeister classes, or, speaking geometrically, the
Reidemeister classes of $\phi$ are cylinders over 
Reidemeister classes of $\overline{\phi}$. Then
$R=R(\phi)= R(\overline{\phi})=\# \Fix (\widehat{\overline{\phi}})$,
because the TBFT is true for an automorphism of a finitely
generated Abelian group (see \cite{FelshB}). If
$\r_1,\dots,\r_R$ are these representations, then 
$\r_1\circ \pi, \dots, \r_R\circ \pi$ are some pairwise non-equivalent $\widehat{\phi}$-fixed  representations of $\G$. 
Then $R(\phi)\le \# \Fix (\widehat{\phi})$. The opposite
inequality is always true (\cite[The proof of Theorem 5.2]{polyc}).

An alternative argument is to see that, since 
the classes are cylinders,
the dimension of the space of shifts
of indicator functions of Reidemeister classes of $\phi$
is the same as for $\overline{\phi}$, and in particular,
is finite. Then the TBFT$_f$ follows from Lemma 3.8 of
\cite{polyc}.
\end{proof}

\def\cprime{$'$} \def\cprime{$'$} \def\cprime{$'$} \def\cprime{$'$}
  \def\cprime{$'$} \def\cprime{$'$} \def\cprime{$'$} \def\cprime{$'$}
  \def\dbar{\leavevmode\hbox to 0pt{\hskip.2ex \accent"16\hss}d}
  \def\cprime{$'$} \def\cprime{$'$}
  \def\polhk#1{\setbox0=\hbox{#1}{\ooalign{\hidewidth
  \lower1.5ex\hbox{`}\hidewidth\crcr\unhbox0}}} \def\cprime{$'$}
  \def\cprime{$'$} \def\cprime{$'$} \def\cprime{$'$}

\end{document}